\documentclass{amsart}
\usepackage[utf8]{inputenc}
\usepackage[unicode]{hyperref}
\usepackage{amssymb,amsmath,mathtools,hyperref}
\usepackage{color}

\newcommand{\R}{\mathbb R}
\DeclarePairedDelimiter\floor{\lfloor}{\rfloor}
\DeclarePairedDelimiter\ceil{\lceil}{\rceil}
\DeclarePairedDelimiter\card{\lvert}{\rvert}
\DeclareMathOperator{\aff}{aff}
\DeclareMathOperator{\conv}{conv}
\DeclareMathOperator{\OT}{OT}
\DeclareMathOperator{\twr}{twr}

\newtheorem{theorem}{Theorem}[section]
\newtheorem{lemma}[theorem]{Lemma}
\newtheorem{prop}[theorem]{Proposition}

\newtheorem{problem}[theorem]{Problem}
\theoremstyle{definition}

\theoremstyle{remark}

\newtheorem{observation}[theorem]{Observation}

\numberwithin{equation}{section}

\title{A note on the tolerated Tverberg theorem}

\author{Natalia García-Colín}
\address[N. García-Colín]{CONACYT, INFOTEC Centro de Investigación e Innovación en Tecnologías de la Información y Comunicación.}
\email{ngarciaco@conacyt.mx}

\author{Miguel Raggi}
\address[M. Raggi]{Escuela Nacional de Estudios Superiores, Unidad Morelia, UNAM}
\email{mraggi@gmail.com}

\author{Edgardo Roldán-Pensado}
\address[E. Roldán-Pensado]{Instituto de Matemáticas, Unidad Juriquilla, UNAM}
\email{e.roldan@im.unam.mx}

\begin{document}

\begin{abstract}
In this paper we give an asymptotically tight bound for the tolerated Tverberg Theorem when the dimension and the size of the partition are fixed. To achieve this, we study certain partitions of order-type homogeneous sets and use a generalization of the Erdős-Szekeres theorem.
\end{abstract}

\maketitle

\section{Introduction}\label{sec:intro}

Tverberg's theorem \cite{Tve1966} states that any set with at least $(d+1)(r-1)+1$ points in $\R^d$ can be partitioned into $r$ disjoint sets $A_1, \ldots, A_r$ such that $\bigcap_{i=1}^r \conv(A_i) \neq \emptyset$. Furthermore, this bound is tight.

The tolerated Tverberg theorem generalizes Tverberg's theorem by introducing a new parameter $t$ called \emph{tolerance}. It states that there is a minimal number $N=N(d,t,r)$ so that any set $X$ of at least $N$ points in $\R^d$ can be partitioned into $r$ disjoint sets $A_1, \ldots, A_r$ such that $\bigcap_{i=1}^r \conv(A_i\setminus Y) \neq \emptyset$ for any $Y\subset X$ with at most $t$ points.

In contrast with the classical Tverberg theorem, the best known bounds for $N(d,t,r)$ are not tight. In \cite{Lar1972}, Larman proved that $N(d,1,2)\leq 2d+3,$ García-Colín showed that $N(d,t,2)\le(t+1)(d+1)+1$ in her PhD thesis \cite{GCol2007}, later published in \cite{GL2015}. This was later generalized by Strausz and Soberón who gave the general bound $N(d,t,r)\le (r-1)(t+1)(d+1)+1$ \cite{SS2012}. Later, Mulzer and Stein gave the bound $N(d,t,r)\le 2^{d-1}(r(t+2)-1)$ which improves the previous bound for $d\le 2$ and is tight for $d=1$ \cite{MS2014}.

As for lower bounds, Ramírez-Alfonsín \cite{RAlf2001} and García-Colín \cite{GL2015}, using oriented matroids, proved that $\ceil{\frac{5d}{3}} +3 \leq N{(d, 1, 2)}$ and $2d + t +1 \leq N{(d, t, 2)},$ respectively. Furthermore, Larman's upper bound is known to be sharp for $d=1, 2, 3$ and $4$ \cite{Lar1972,Forge2001}. Lastly, Soberón gave the bound $r(\floor{\frac d2}+t+1)\le N(d,t,r)$ \cite{Sob2015}.

In this paper we show that for fixed $d$ and $r$, the correct value for $N(d,t,r)$ is asymptotically equal to $rt$. To be precise, we prove the following theorem.
\begin{theorem}\label{thm:main}
For fixed $r$ and $d$ we have that
$$N(d,t,r) = rt + o(t).$$
\end{theorem}
This improves all previously known upper bounds whenever $t$ is large compared to $r$ and $d$, and comes with a matching lower bound.

The proof follows from studying the behavior of $t$ with respect to $N$ and using the Erdős-Szekeres theorem for cyclic polytopes in $\R^d$. We include a short review of cyclic polytopes and the Erdős-Szekeres theorem in Section \ref{sec:prelim}. In \ref{sec:alternating} we prove a useful Lemma about alternating partitions of a cyclic polytope which leads to an interesting open problem. The proof of Theorem \ref{thm:main} is detailed in Section \ref{sec:remarks}.

\section{Preliminaries}\label{sec:prelim}

In this section we introduce some definitions and recall some well known concepts which we later use in the proofs of this paper.

\subsection{Order-type homogeneous sets}

Any ordered set $X \subset \R^d$ with the property that the orientation of any ordered subset of $X$ with $(d+1)$ elements is always the same is called an \emph{order-type homogeneous set}. 
 
A classic example of such a set is the set of vertices of a cyclic polytope, $X$, which is constructed as follows: consider the moment curve $\gamma(\alpha)=(\alpha, \alpha^2, \ldots, \alpha^d)$, given real numbers $\alpha_1<\alpha_2<\dots<\alpha_n$ define $X=\{\gamma(\alpha_1),\gamma(\alpha_2),\dots,\gamma(\alpha_n)\}.$ The set $\conv(X)$ is the $d$-dimensional \emph{cyclic polytope} on $n$ points and any other polytope combinatorially equivalent to the cyclic polytope is also sometimes referred to as a cyclic polytope or, more generally, as an order-type homogeneous set.

Order-type homogeneous sets have been studied extensively \cite{Bis2006,Gal1963,Gr2003,Mat2002,Zie1995} and have proven to be very useful as examples with extremal properties in various combinatorial problems. In our case they will prove useful in finding better bounds for the tolerated Tverberg number $N(d,t,r)$. 

The following lemma, due to Gale \cite{Gal1963} is one of the most useful tools for studying the properties of order-type homogeneous sets.

\begin{lemma}[Gale's evenness criterion]\label{lem:Gale}
Let $X=\{x_1,x_2,\ldots,x_n\}$ be an order-type homogeneous set. A subset $F \subset X$ such that $\card F=d$ determines a facet of $conv(X)$ if and only if, any two vertices in $ X \setminus F$ have an even number of vertices of $F$ between them in the order.
\end{lemma}

As a consequence of Lemma \ref{lem:Gale}, the polytopes that arise as the convex hulls of order-type homogeneous sets are known to be $\floor{\frac{d}{2}}$-neighborly. That is, the convex hull of every $\floor{\frac{d}{2}}$ points in $X$ is contained in a facet of $C$ and, since $C$ is simplicial, the convex hull of such vertices is a $\floor{\frac{d}{2}}-1$ face of $C$.

Another useful fact when working with order-type homogeneous sets is Lemma 2.1 from \cite{BMP2014}:

\begin{lemma}\label{lem:ot}
An ordered set $X=\{x_1,x_2,\dots x_n\}$ in general position in $\R^d$ is order-type homogeneous if and only if the polygonal path $\pi=x_1x_2\dots x_n$ intersects every hyperplane in at most $d$ points, with the exception of the hyperplanes that contain an edge of $\pi$.
\end{lemma}

\subsection{The Erdős-Szekeres theorem}
In 1935 Erdős and Szekeres proved two important theorems in combinatorial geometry \cite{ES1935}.
The first Erdős-Szekeres theorem implies that any sequence of numbers with length $(n-1)^2+1$ always contains a monotonous (either increasing or decreasing) subsequence. The second Erdős-Szekeres theorem states that among any $2^{\Theta(n)}$ points in the plane there are $n$ of them in convex position.

These two theorems can be thought as results on order-type homogeneous sets in dimensions $1$ and $2$.
The following theorem, proved in \cite{Suk2014,BMP2014}, generalizes both results to order-type homogeneous sets in any $d$-dimensional space.

\begin{theorem}\label{thm:Suk}
Let $\OT_d(n)$ be the smallest integer such that any set of $\OT_d(n)$ points in general position in $\R^d$ contains an order-type homogeneous subset of size $n$. Then $\OT_d(n)=\twr_d(\Theta(n))$, where the tower function $\twr_d$ is defined by $\twr_1(\alpha)=\alpha$ and $\twr_{i+1}(\alpha)=2^{\twr_i(\alpha)}$.
\end{theorem}

\section{Tolerance of partitions of sets}\label{sec:proof}

Let $X=\{x_1,x_2,\dots,x_n\}$ be a set of points in $\R^d$. We define the \emph{tolerance} $t(X,r)$ of $X$ as the maximum number of points such that there is a partition $A_1, \ldots, A_r$ of $X$ with the property that $\bigcap_{i=1}^r \conv(A_i\setminus Y) \neq \emptyset$ for any $Y\subset X$ with at most $t(X,r)$ points.

\begin{observation}\label{dumbpigeonholelemma}
	Let $X_1$, $X_2$ be disjoint sets of points of $\R^d$. Then $t(X_1\cup X_2,r) \geq t(X_1,r)+t(X_2,r)$.
\end{observation}

We also define the following two numbers;
$$t(n, d, r)=\min_{\substack{ X \subset \R^d \\ |X|=n}}\{t(X,r)\} \quad \text{ and } \quad T(n, d, r)=\max_{\substack{X \subset \R^d \\ |X|=n}}\{t(X,r)\}.$$
For fixed $n,$ $t(n, d, r)$ indicates that \emph{there exists a set} $X$ such that for all partitions $A_1, \ldots, A_r$ of $X,$ such that $\bigcap_{i=1}^r \conv(A_i\setminus Y) \neq \emptyset$ for any $Y\subset X$ with at most $t(X,r)=t(n, d, r)$ points, while $T(n, d, r)$ indicates that \emph{for all} $X$ with size $n$ there is a partition $A_1, \ldots, A_r$ of $X,$ such that $\bigcap_{i=1}^r \conv(A_i\setminus Y) \neq \emptyset$ for any $Y\subset X$ with at most $t(X,r)=T(n, d, r)$ points.

\subsection{Tolerance of order-type homogeneous sets}\label{sec:alternating}

For proving Theorem \ref{thm:main} we need to study a specific type of partitions. Let $X=\{x_1,x_2,\dots,x_n\}$ be an ordered set of points in $\R^d$ with a the order specified by the subindices and let $r>0$ be a fixed integer. The partition of $X$ into $r$ sets $A_1, \ldots, A_r$ given by $A_i=\{x_j : j \equiv i \mod r\}$ is called an \emph{alternating partition}. Our main interest is to determine when the convex hulls of the sets $A_i$ have a common point and how \emph{tolerant} are they are.

\begin{lemma} \label{lem:alternating}
Let $X=\{x_1,x_2,\dots,x_n\}$ be an order-type homogeneous set of points in $\R^d$ with alternating partition $A_1, \ldots, A_r$. Then there is a number $c(d,r)\le (d+1)(\floor{\frac d2}+1)(r-1)+1\approx \frac{rd^2}{2}$ such that if $n\ge c(d,r)$, then $\bigcap_{i=1}^r \conv(A_i)\neq\emptyset$.
\end{lemma}

\begin{figure}
\includegraphics{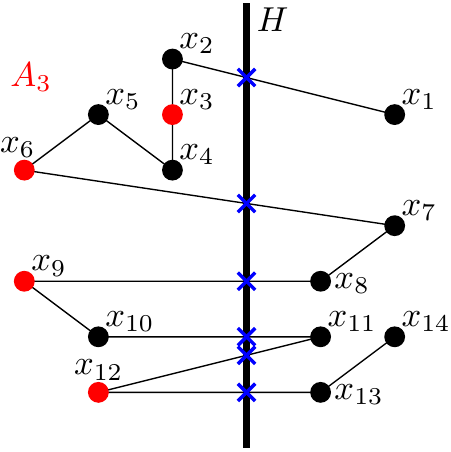}
\caption{An example for Lemma \ref{lem:alternating} with $n=14$, $d=6$ and $r=3$. The set $A_3$, in red, is to the left of $H$ and the path $\pi$ intersects $H$ at most $d$ times.}
\label{fig:lem}
\end{figure}

\begin{proof}
Let $O$ be a center point for $X$. This means that every semi-space containing $O$ also contains at least $\ceil{\frac{n}{d+1}}$ points of $X$. We will show that $O\in\conv(A_i)$ for every $i$. Suppose this is not the case. Then there is a hyperplane $H$ strictly separating $O$ from some $\conv(A_i)$. We may assume (by perturbing $H$ if necessary) that no point in $X$ is contained in $H$.

Let $H^+$ be the semi-space bounded by $H$ that contains $O$. Since $O$ is a center point then $X\cap H^+$ contains at least $\ceil{\frac{n}{d+1}}>\left(\floor{\frac d2}+1\right)(r-1)$ points.

On the other hand, by Lemma \ref{lem:ot}, the polygonal path $\pi$ generated by $X$ intersects $H$ at most $d$ times. Therefore $\pi\cap H^+$ has at most $\floor{\frac d2}+1$ connected components and, since $A_i\cap H^+=\emptyset$, each of these components is a sub-path of $\pi$ contained between two consecutive points of $A_i\subset\pi$ (see Figure \ref{fig:lem}). Thus, each component contains at most $r-1$ points from $X$, so $X\cap H^+$ has at most $(\floor{\frac d2}+1)(r-1)$ points. This contradicts our assumption that $O\not\in\conv(A_i)$.
\end{proof}

The bound for $c(d,r)$ given in the previous lemma is not tight. In fact it can be improved when $d$ is even by noticing that, if $n\equiv i$ (mod $r$), then $X\cap H^+$ can have at most $\frac d2(r-1)+i$ points. The bound obtained in this case is $c(d,r)\le \min_i \left\{\frac{d(d+1)}2(r-1)+i(d+1)+s_i\right\}$, where $s_i$ be the smallest positive integer such that $s_i\equiv \frac{d(d+1)}2-id$ (mod $r$). When $r$ is large compared to $d$ this simply equals $\frac{d(d+1)}{2}r$. However this bound is still not tight, giving rise to an interesting open question.

\begin{problem}
Determine the smallest value for $c(d,r)$ for which Lemma \ref{lem:alternating} holds.
\end{problem}

The cases $d=1$ and $d=2$ are not difficult. We have the following values: $c(1,r)=2r-1$, $c(2,1)=1$, $c(2,2)=4$ and $c(2,r)=3r$ when $r\ge 3$. Note that the bound from Lemma \ref{lem:alternating} is tight for $d=1$ and the bound described for even dimensions after the proof of the lemma is tight for $d=2$.

If $r=2$, a simple separating-hyperplane argument shows that $c(d,r)=d+2$.
In general it can also be proved that $c(d,r)\ge (d+1)r$ whenever $r>d$, but this is also not tight. The example in Figure \ref{fig:ex} shows that $c(3,4)>16$.

\begin{figure}
\includegraphics[width=0.5\textwidth]{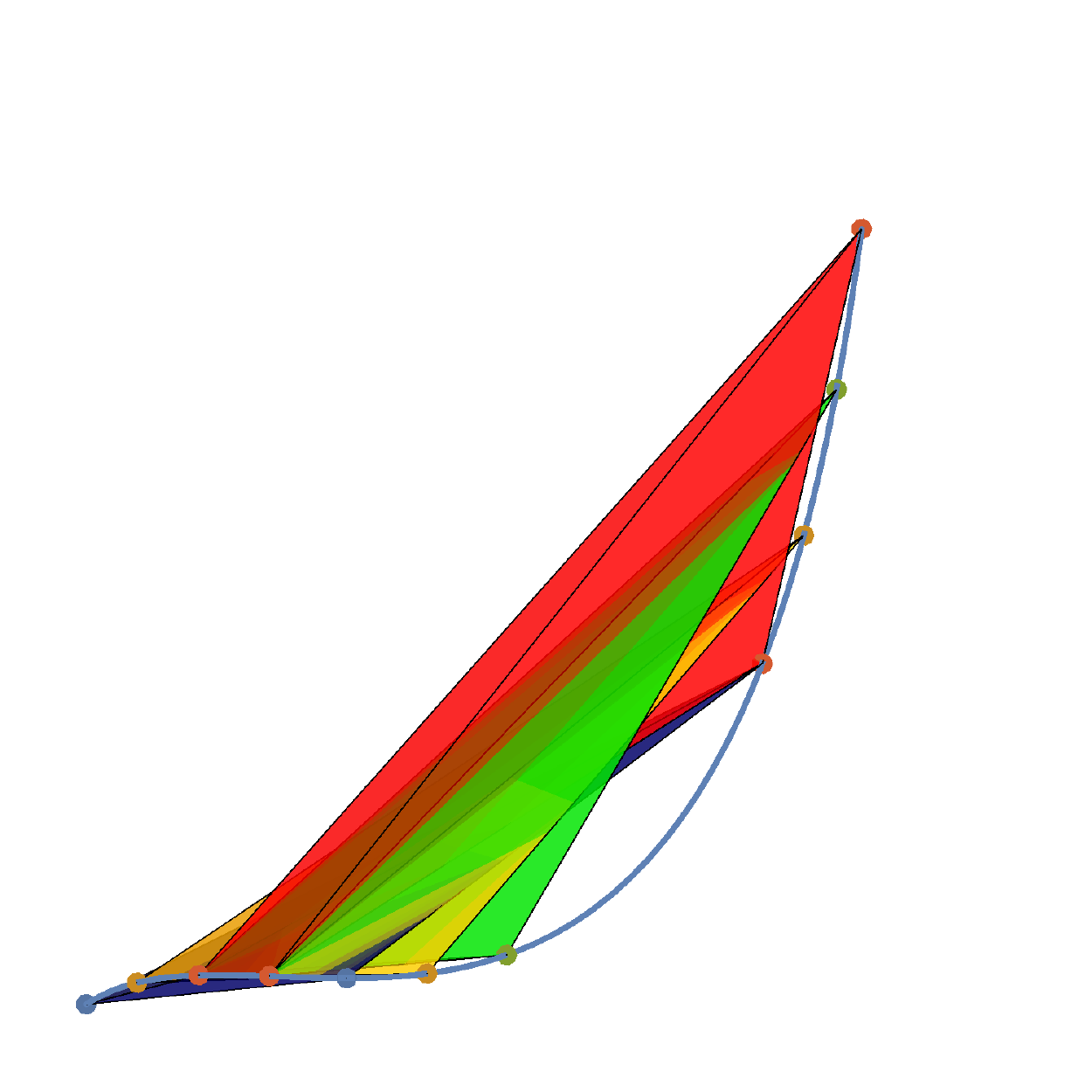}
\caption{Four alternating tetrahedra with vertices on the moment curve $(t,t^2,t^3)$ and without a common point. In this example $t$ takes the values $-4$, $-3$, $-2$, $-2$, $-2$, $-1$, $-1$, $-1$, $0$, $1$, $2$, $6$, $6$, $7$, $8$ and $9$, which may be perturbed so that all vertices are distinct.}
\label{fig:ex}
\end{figure}

Now we are ready to study the tolerance of an order-type homogeneous set. The upper bound in the following theorem was proved by Soberón in \cite{Sob2015} but we include the proof for completion.

\begin{theorem}\label{thm:cotastol}
Let $X=\{x_1,x_2,\dots,x_n\}$ be an order-type homogeneous set of points in $\R^d$. Then $\floor{\frac{n}{r}} - c(d,r) \le t(X,r)\leq \floor{\frac{n}{r}}-\floor{\frac{d}{2}}$, where $c(d,r)$ is the number from Lemma \ref{lem:alternating}.
\end{theorem}

\begin{proof}
First we prove the upper bound for $t(X,r)$. For any partition of $X$ into $r$ disjoint parts $A_1, \ldots, A_r$ we will find that for some $i \in [r]=\{1, \ldots, r\}$, $\card{A_i}\leq \floor{\frac{n}{r}}$. Let $Y \subset A_i $ be any subset such that $\card{A_i\setminus Y}\leq \floor{\frac{d}{2}}$, then necessarily $\conv(A_i\setminus Y)$ is disjoint from $\conv(X\setminus A_i)$ as $X$ is the set of vertices of a $\floor{\frac{d}{2}}$-neighborly polytope and therefore $\conv(A_i\setminus Y)$ is a $\floor{\frac{d}{2}}-1$ face of the polytope $\conv(X)$. In particular, this implies that $\bigcap_{i=1}^r \conv(A_i \setminus Y) = \emptyset$, hence $t(X,r)\leq \floor{\frac{n}{r}}-\floor{\frac{d}{2}}$.

For the lower bound, consider the alternating partition $A_1, \ldots, A_r$ of $X$. Assume that $Y \subset X$ satisfies $\card Y \leq \floor{\frac{n}{r}}-c(d,r)$. By the pigeonhole principle, we can find $X' \subset X\setminus Y$ such that $\card{X'}= c(d,r)$ and the restriction of the partition of $X$ to $X'$ (i.e. $A_1\cap X', \ldots, A_r \cap X'$) is also an alternating partition. Thus, by Lemma \ref{lem:alternating} we have that $\bigcap_{i=1}^r \conv(A_i \cap X') \neq \emptyset$ and the theorem follows.
\end{proof}

\subsection{Tolerance of partitions of general sets}

The tolerated Tverberg number's bound $N(d,t,r)\leq (r-1)(t+1)(d+1)+1$ implies that for any set $X$ of $n$ points, its tolerance is bounded by $\frac{n-1}{(r-1)(d+1)}-1 \leq t(X, r).$ On the other hand, we can argue that the tolerance under any partition of a set can never be greater than the size of the smallest part in the partition, \textit{i.e.} $T(n, d, r)\leq \floor{\frac{n}{r}}.$ 

The arguments in the previous paragraphs imply that $\frac{n-1}{(r-1)(d+1)}-1 \leq t(n, d, r) \leq t(X,n) \leq T(n, d, r) \leq \floor{\frac{n}{r}}$ holds for any $X \subset \R^d$ with $\card{X}=n$.

In this section we exhibit improved bounds for the tolerance of partitions of general sets.

\begin{prop}\label{lem: bound2}
For any positive integers $n,d,r$ we have that $T(n,d,r)\leq \floor{\frac{n}{r}}-\floor{\frac{d}{2}}$.
\end{prop}

\begin{proof} Let $A_1, \ldots A_r$ be a partition of the set, and let $t'$ be maximum such that $\bigcap_{i=1}^r \conv(A_i\setminus Y) \neq \emptyset$ for any $Y\subset X$ with at most $t'$ points. Then $t'\leq T(n, d, r).$

Let $A_i, A_j$ be parts such that $i \neq j$, we may assume that $|A_i \cup A_j|\geq d+2,$ otherwise $t'=0.$ Then for any subset of $d$ points in $A_i \cup A_j$, $D$ we must have that the hyperplane $H=\aff(D)$ is such that;
$\card{ H^+ \cap A_i} + \card{H^- \cap A_j} >t'$ and $\card{H^- \cap A_i} + \card{H^+ \cap A_j} >t'$.

Hence $\card{A_i} + \card{A_j} -d > 2t'$ and adding through all the different pairs, 
$ \sum_{i<j} \card{A_i} + \card{A_j}> \binom{r}{2} (2t'+d)$. That is, 
$ (r-1)\sum_{i \in [r]} \card{A_i} > \binom{r}{2} (2t'+d)$ and thus 
$n >\frac{r}{2}(2t'+d)$. Rearranging the later equation we can obtain 
$\frac{n}{r}>t' +\frac{d}{2}$. Therefore 
$\frac{n}{r}>t' + \floor{\frac{d}{2}}$ and so
$\floor{\frac{n}{r}} \geq t' + \floor{\frac{d}{2}}$. 
\end{proof}


\begin{lemma} \label{thm:main2} Let $r,d$ be fixed natural numbers. For a large enough $n$ we have that $t(n,d,r) \geq \frac{n}{r}-o(n)$.
\end{lemma}

%
%
%

\begin{proof} Fix small $\varepsilon > 0$. We shall construct a large number $n$ satisfying that, for any set $X$ of $n$ points in $\R^d$, we have $t(X,r) \geq \frac{n}{r}(1-\varepsilon)$. Let $c=c(d,r)$ as in Lemma \ref{lem:alternating}.

Assume that $n = \OT_d(k) + mk$ for some positive integers $m$ and $k$, where $\OT_d$ is the bound from Theorem \ref{thm:Suk}. Then, given a set $X$ of $n$ points in general position in $\R^d$, we can select $m$ pairwise-disjoint order-type homogeneous subsets $X_1, X_2,\dots, X_m$ of size $k$ from $X$.

Partition the points of each $X_i$ into $r$ parts using the alternating method proposed in Section \ref{sec:alternating}. By Theorem \ref{thm:cotastol}, we have that $t(X_i,r) \geq \frac{k}{r} - c$ and therefore, by Observation \ref{dumbpigeonholelemma}, $t(X,r) \geq t(X_1,r)+\dots+t(X_m,r) \geq m\left(\frac kr-c\right)$. We may rewrite this last value as
$$m\left(\frac kr-c\right)=\frac nr\left(\frac{mk-mcr}{n}\right)=\frac nr\left(1-\frac{\OT_d(k)+mcr}{\OT_d(k)+mk}\right).$$

By choosing a large enough $k$ so that $\frac{1+cr}{1+k} < \varepsilon$ and $m=\OT_d(k)$, we obtain $t(X,r)\ge\frac nr\left(1-\frac{1+cr}{1+k}\right)>\frac nr(1-\varepsilon)$.
\end{proof}

\section{Bounds on the Tolerated Tverberg number}\label{sec:remarks}

So far we have being concerned with studying the behavior of $t$ with respect to $n, d$ and $r$. By a simple manipulation of the results in the previous section, we may now easily prove Theorem \ref{thm:main}.

\begin{proof}[Proof of Theorem \ref{thm:main}]
Fix $r$ and $d$. By Proposition \ref{lem: bound2} we have that $t \leq \floor{\frac{n}{r}}-\floor{\frac{d}{2}}$, which implies $n\ge tr+\frac{r(d-2)}{2}$. Lemma \ref{thm:main2} can be rewritten as $n\leq tr + o(n)$.
These inequalities imply $n=\Theta(t)$, so we have that
$$rt+\frac{r(d-2)}{2}\leq n \leq rt+ o(t),$$
which yields the result.
\end{proof}

This result clarifies why the search for a definite $N(d, r, t)$ has been elusive. It seems that the relationship between $t$ and $N$ changes as $t$ increases, as opposed to being a constant multiple of $t$ (for a fixed $d$ and $r$).

From the analysis made in Lemma \ref{thm:main2} it follows that the term $o(t)$ in Theorem \ref{thm:main} decays like $\frac{t}{\log^{(d)}(t)}$, which is extremely slow.
It is our impression that $N(d,t,r)$ approaches $rt$ much faster than this.


\end{document}